\documentclass[11pt,reqno]{amsart}

\usepackage{amscd,amssymb,amsmath,amsthm}
\usepackage[arrow,matrix]{xy}
\usepackage{graphicx}
\usepackage{color}
\usepackage{cite}
\newtheorem{thm}{Theorem}

\newtheorem{lemma}{Lemma}
\newtheorem{pro}{Proposition}
\newtheorem{rk}{Remark}
\newtheorem{cor}{Corollary}

\numberwithin{equation}{section} 

\begin{document}
\title[$q$-state $p$-adic Potts model]{
On $q$-state $p$-adic Potts model on a Cayley tree}

\author{O. N. Khakimov}

\address{O.N. Khakimov\\ Institute of mathematics,
29, Do'rmon Yo'li str., 100125, Tashkent, Uzbekistan.}
\email {hakimovo@mail.ru}

\begin{abstract} This work is devoted to description of all
translation-invariant
$p$-adic Gibbs measures (TIpGMs) for the $q$-state Potts
model on a Cayley tree
of order $k$. In particular, for the Cayley tree
of order three we give exact number of such measures.
Moreover we give criterion of boundedness of TIpGMs
\end{abstract}
\maketitle


\section{introduction}

The $p$-adic numbers were first introduced by the German
mathematician K.Hensel. They form an integral part of number
theory, algebraic geometry,
representation theory and other branches of modern mathematics.
However, numerous applications of these numbers to
theoretical physics have been proposed \cite{14}
to quantum mechanics and to $p$-adic valued physical
observable \cite{k91}. A number of $p$-adic models in physics
cannot be described using ordinary probability theory based on the
Kolmogorov axioms.

In \cite{33} a theory of stochastic processes with values in
$p$-adic and more general non-Archimedean fields was developed,
having probability distributions with non-Archimedean values.

One of the basic branches of mathematics lying at the base of the
theory of statistical mechanics is the theory of probability and
stochastic processes. Since the theories of probability and
stochastic processes in a non-Archimedean setting have been
introduced, it is natural to study problems of statistical
mechanics in the context of the $p$-adic theory of probability.

We note that $p$-adic Gibbs measures were studied for several
$p$-adic
models of statistical mechanics \cite{GRR,16,TMP2,KMR},
\cite{MRasos,37,MNGM,Fg,MFHA,RK,RKh,bookroz}.

The Potts model is one of the most studied models in statistical
mechanics, since this model is related to a number of outstanding
problems in statistical and mathematical physics, and graph theory
\cite{GR}. Originally, the Potts model was introduced as a
generalization of the Ising model to more than two spin values. It
is known that \cite{MRasos} there exist phase transition for the
$q$-state $p$-adic Potts model on the Cayley tree of order $k$ if
and only if $q\in p\mathbb N$. In \cite{RKh} it was fully
investigated translation-invariant $p$-adic Gibbs
 measures for the $q$-state Potts model on a Cayley tree of order two.
 In this paper, for some conditions to $k$ and prime number $p$,
we shall fully describe the set of
TIpGMs for the $q$-state Potts model on a Cayley tree of order $k$.
Moreover, we shall show that
there occurs a phase transition for the $p$-adic Potts model on a
Cayley tree of order three.

\section{definitions and preliminary results}

\subsection{\bf $p$-adic numbers and measures.}

Let $\mathbb Q_p$ be the field of $p$-adic numbers with norm $|\cdot|_p$.
Recall that $\mathbb Q_p$ is a non-Archimedean (ultra-metric) space,
i.e. its norm satisfies the strong triangle inequality.
We often use the following property ($p$-adic norm's property):
$$|x\pm y|_p=\max\{|x|_p, |y|_p\},\qquad\mbox{if }\ |x|_p\neq|y|_p.$$
It is known \cite{29,48} that any $p$-adic number $x\ne 0$
can be uniquely represented
in the canonical form
\begin{equation}\label{ek}
x = p^{\gamma(x)}(x_0+x_1p+x_2p^2+\dots),
\end{equation}
where $\gamma=\gamma(x)\in \mathbb Z$ and the integers $x_j$
satisfy: $x_0 > 0$,
$0\leq x_j \leq p - 1$.

We recall that $Z_p=\{x\in\mathbb Q_p:\ |x|_p\leq1\}$ and
$Z_p^*=\{x\in\mathbb Q_p:\ |x|_p=1\}$
are the set of all $p$-adic integers
and units of $\mathbb Q_p$, respectively.
A number $a\in\mathbb Q_p\setminus\{0\}$ is called a $k$-residue modulo
$p$ if
there exists a number $x\in\mathbb Q_p$ such that
$x^k\equiv a(\operatorname{mod} p)$.
Otherwise, $a$ is called a $k$-nonresidue modulo $p$.

\begin{thm}\label{tx3}\cite{48}
The equation
$x^2 = a$, $0\ne a =p^{\gamma(a)}(a_0 + a_1p + ...), 0\leq a_j
\leq p - 1$, $a_0 > 0$ has a solution $x\in \mathbb Q_p$ iff
hold true the following:\\
$(a)$ $\gamma(a)$ is even;\\
$(b)$ $y^2=a_0(\operatorname{mod} p)$ is solvable for $p\ne 2$;
the equality $a_1=a_2=0$ holds if $p=2$.
\end{thm}

\begin{thm}\label{txq}\cite{COR}
Let $k>2$ is a integer number and $(k,p)=1$. The equation
$x^k=a$ with $a\neq0$ has a solution
$x\in\mathbb Q_p$ if and only if\\
$(a)$ $k$ divides $\gamma(a)$, and,\\
$(b)$ $a_0$ is a $k$-residue modulo $p$.
\end{thm}

\begin{thm}\label{tx2}\cite{LOR}
Let $k=mp$ and $(m,p)=1$. The equation $x^k=a$ with $a\neq0$
has a solution
$x\in\mathbb Q_p$ if and only if\\
$(a)$ $k$ divides $\gamma(a)$, and,\\
$(b)$ $a_0$ is a $m$-residue modulo $p$ and
$a_0^p\equiv a_0+a_1p(\operatorname{mod} p^2)$.
\end{thm}

\begin{lemma} $($ Hensel's lemma, \cite{48}$)$ Let $f(x)$ be
polynomial whose the coefficients
are $p$-adic integers. Let $a_0$ be a $p$-adic integer such that
for some $i\geq0$ we have
\[
\begin{array}{ll}
f(a_0)\equiv0(\operatorname{mod }p^{2i+1}),\\
f^\prime(a_0)\equiv0(\operatorname{mod }p^{i}),\qquad
f^\prime(a_0)\not\equiv0(\operatorname{mod }p^{i+1}).
\end{array}\]
Then $f(x)$ has a unique $p$-adic integer root $a$ which satisfies
$a\equiv a_0(\operatorname{mod }p)$.
\end{lemma}
Denote
$$
\mathcal E_p=\left\{x\in\mathbb Z_p^*: |x-1|_p<p^{-1/(p-1)}\right\}.
$$

A more detailed description of $p$-adic calculus and $p$-adic
mathematical physics can be found in \cite{29,48}.

Let $(X,{\mathcal B})$ be a measurable space, where ${\mathcal B}$
is an algebra of subsets of $X$. A function $\mu: {\mathcal B}\to
\mathbb Q_p$ is said to be a $p$-adic measure if for any $A_1, . . .
,A_n\in {\mathcal B}$ such that $A_i\cap A_j = \emptyset$, $i\ne
j$, the following holds:
$$\mu(\bigcup^n_{j=1}A_j)=\sum^n_{j=1}\mu(A_j).$$
A $p$-adic measure is called a probability measure if $\mu(X) =
1$. A $p$-adic probability measure $\mu$ is called {\it bounded}
if $\sup\{|\mu(A)|_p:A\in\mathcal B\}<\infty$ (see, \cite{k91}).

We call a $p$-adic measure a probability measure \cite{16} if $\mu(X)=1$.

\subsection{\bf Cayley tree.}

The Cayley tree $\Gamma^k$
of order $ k\geq 1 $ is an infinite tree, i.e., a graph without
cycles, such that exactly $k+1$ edges originate from each vertex.
Let $\Gamma^k=(V, L)$ where $V$ is the set of vertices and  $L$ the
set of edges.
Two vertices $x$ and $y$ are called {\it nearest neighbors} if there
exists an
edge $l \in L$ connecting them.
We shall use the notation $l=\langle x,y\rangle$.
A collection of nearest neighbor pairs
$\langle x,x_1\rangle, \langle x_1,x_2
\rangle,...,\langle x_{d-1},y\rangle$ is called a {\it
path} from $x$ to $y$. The distance $d(x,y)$ on the
Cayley tree is the number of
 edges of the shortest path from $x$ to $y$.

For a fixed $x^0\in V$, called the root, we set
\begin{equation*}
W_n=\{x\in V\,| \, d(x,x^0)=n\}, \qquad V_n=\bigcup_{m=0}^n W_m
\end{equation*}
and denote
$$
S(x)=\{y\in W_{n+1} :  d(x,y)=1 \}, \ \ x\in W_n, $$ the set  of
{\it direct successors} of $x$.

Let $G_k$ be a free product of $k + 1$ cyclic groups of the
second order with generators $a_1, a_2,\dots, a_{k+1}$,
respectively.
It is known that there exists a one-to-one correspondence
between the set
 of vertices $V$ of the Cayley tree $\Gamma^k$ and the group $G_k$.

\subsection{\bf $p$-adic Potts model}

Let $\mathbb Q_p$ be the field of $p$-adic numbers and
$\Phi$ be a finite set. A configuration $\sigma$ on $V$ is
then defined
as a function $x\in V\to\sigma(x)\in\Phi$; in a similar fashion
one defines a configuration $\sigma_n$ and $\sigma^{(n)}$ on $V_n$
and $W_n$ respectively. The set of all configurations on $V$
(resp. $V_n,\ W_n$) coincides with $\Omega=\Phi^V$
(resp.$\Omega_{V_n}=\Phi^{V_n},\ \Omega_{W_n}=\Phi^{W_n}$). Using
this, for given configurations $\sigma_{n-1}\in\Omega_{V_{n-1}}$
and $\sigma^{(n)}\in\Omega_{W_n}$ we define their concatenations
by
$$
(\sigma_{n-1}\vee\sigma^{(n)})(x)=\left\{\begin{array}{ll}
\sigma_{n-1}(x),& \text{if}\  x\in V_{n-1},\\
\sigma^{(n)}(x),& \text{if}\  x\in W_n.
 \end{array}\right.
$$
It is clear that $\sigma_{n-1}\vee\sigma^{(n)}\in\Omega_{V_n}.$\\

Let $G^*_k$ be a subgroup of the group $G_k$. A function $h_x$
(for example, a configuration $\sigma(x)$) of $x\in G_k$
is called
$G^*_k$-periodic
if $h_{yx}=h_x$ (resp. $\sigma(yx)=\sigma(x)$) for any
$x\in G_k$ and
$y\in G^*_k.$

A $G_k$-periodic function is called {\it translation-invariant}.

We consider {\it $p$-adic Potts model} on a Cayley tree,
where the spin takes values in the set
$\Phi:=\{1,2,\dots,q\}$, and is assigned to the vertices
of the tree.

The (formal) Hamiltonian of $p$-adic Potts model is
\begin{equation}\label{ph}
H(\sigma)=J\sum_{\langle x,y\rangle\in L}
\delta_{\sigma(x)\sigma(y)},
\end{equation}
where $J\in B(0, p^{-1/(p-1)})$ is a coupling constant,
$\langle x,y\rangle$ stands for nearest neighbor vertices
and $\delta_{ij}$ is the Kroneker's
symbol:
$$\delta_{ij}=\left\{\begin{array}{ll}
0, \ \ \mbox{if} \ \ i\ne j\\[2mm]
1, \ \ \mbox{if} \ \ i= j.
\end{array}\right.
$$

\subsection{\bf $p$-adic Gibbs measure}
Define a finite-dimensional distribution of a $p$-adic
probability
 measure $\mu$ in the volume $V_n$ as
\begin{equation}\label{p*}
\mu_{\tilde h}^{(n)}(\sigma_n)=Z_{n,\tilde h}^{-1}\exp_p\left\{H_n(\sigma_n)
+\sum_{x\in W_n}{\tilde h}_{\sigma(x),x}\right\},
\end{equation}
where  $Z_{n,\tilde h}$ is the normalizing factor,
$\{{\tilde h}_x=({\tilde h}
_{1,x},\dots, {\tilde h}_{q,x})\in\mathbb Q_p^q, x\in V\}$
is a
collection of vectors and
$H_n(\sigma_n)$ is the restriction of Hamiltonian on $V_n$.

We say that the $p$-adic probability distributions (\ref{p*})
are
compatible if for all
$n\geq 1$ and $\sigma_{n-1}\in \Phi^{V_{n-1}}$:
\begin{equation}\label{us}
\sum_{\omega_n\in \Phi^{W_n}}\mu_{\tilde h}^{(n)}
(\sigma_{n-1}\vee \omega_n)
=\mu_{\tilde h}^{(n-1)}(\sigma_{n-1}).
\end{equation}
Here $\sigma_{n-1}\vee \omega_n$ is the concatenation of the
configurations.\\
We note that an analog of the Kolmogorov extension theorem for
distributions can be proved for $p$-adic distributions given by (\ref{p*})
(see \cite{16}). So, if for some function $\tilde h$ the measures
$\mu_{\tilde h}^{(n)}$ satisfy
the compatibility condition (\ref{us}), then there is a unique $p$-adic
probability measure, which we denote by $\mu_{\tilde h}$, since it
depends on $\tilde h$.
Such a measure is called a {\it $p$-adic Gibbs measure} (pGM)
corresponding
to the Hamiltonian (\ref{ph}) and vector-valued function
${\tilde h}_x, x\in V$. The set of all
$p$-adic Gibbs measures associated with functions
$\tilde h=\{\tilde h_x,\ x\in V\}$ is
denoted by $\mathcal G(H)$. If there are at least two distinct $p$-adic
Gibbs measures $\mu,\nu\in\mathcal G(H)$ such that $\mu$ is
bounded and $\nu$ is
unbounded, then it is said that {\it a phase transition occurs}.
By another words,
one can find two different functions $s$ and $h$ defined on $V$
such that there exist
the corresponding measures $\mu_s$ and $\mu_h$, for which one is
bounded,
another one
is unbounded. Moreover, if there is a sequence of sets $\{A_n\}$
such that
$A_n\in\Omega_{V_n}$ with $|\mu(A_n)|_p\to0$ and
$|\nu(A_n)|_p\to\infty$
as $n\to\infty$,
then we say that there occurs a {\it strong phase transition}
\cite{MNGM,MFHA}.

The following statement describes conditions on ${\tilde h}_x$
guaranteeing
compatibility of $\mu_{\tilde h}^{(n)}(\sigma_n)$.

\begin{thm}\label{ep} \label{Theorem1} (see \cite[p.89]{MRasos})
The $p$-adic probability distributions
$\mu_n(\sigma_n)$, $n=1,2,\ldots$, in
(\ref{p*}) are compatible for Potts model iff for any $x\in V\setminus\{x^0\}$
the following equation holds:
\begin{equation}\label{p***}
 h_x=\sum_{y\in S(x)}F(h_y,\theta),
\end{equation}
where $F: h=(h_1, \dots,h_{q-1})\in\mathbb Q_p^{q-1}\to F(h,\theta)=
(F_1,\dots,F_{q-1})\in\mathbb Q_p^{q-1}$ is defined as
$$F_i=\log_p\left({(\theta-1)\exp_p(h_i)+\sum_{j=1}^{q-1}\exp_p(h_j)+
1\over \theta+ \sum_{j=1}^{q-1}\exp_p(h_j)}\right),$$
$\theta=\exp_p(J)$, $S(x)$ is the set of direct successors of
$x$ and $h_x=\left(h_{1,x},\dots,h_{q-1,x}\right)$ with
\begin{equation}\label{hh}
h_{i,x}={\tilde h}_{i,x}-{\tilde h}_{q,x}, \ \ i=1,\dots,q-1.
\end{equation}
\end{thm}

From Theorem \ref{ep} it follows that for any $h=\{h_x,\ \ x\in V\}$
satisfying (\ref{p***}) there exists a unique pGM $\mu_h$ for the
$p$-adic Potts model.

\begin{thm}\label{qpn}\cite{MRasos} If $q\not\in p\mathbb N$
then there exists
a unique $p$-adic Gibbs measure $\mu_0$ for the $p$-adic Potts
model
(\ref{ph}) on a Cayley tree of order $k$. Moreover, a measure
$\mu_0$ is
translation-invariant.
\end{thm}

\section {translation-invariant $p$-adic gibbs measures for the potts model.}

In this section, we consider $p$-adic Gibbs measures which
are translation-invariant,
 i.e., we assume $h_x=h=(h_1,\dots,h_{q-1})\in\mathbb Q_p^{q-1}$ for all
 $x\in V$.  Then from equation (\ref{p***}) we get $h=kF(h,\theta)$, i.e.,
\begin{equation}\label{pt}
h_i=k\log_p\left({(\theta-1)\exp_p(h_i)+\sum_{j=1}^{q-1}\exp_p(h_j)+
1\over \theta+ \sum_{j=1}^{q-1}\exp_p(h_j)}\right),\ \ i=1,\dots,q-1.
\end{equation}
Denoting $z_i=\exp_p(h_i), i=1,\dots,q-1$, we get from (\ref{pt})
\begin{equation}\label{pt1}
z_i=\left({(\theta-1)z_i+\sum_{j=1}^{q-1}z_j+1\over \theta+
\sum_{j=1}^{q-1}z_j}\right)^k,\ \ i=1,\dots,q-1.
\end{equation}
Note that for a solution $z=(z_1,...,z_{q-1})$ of the system of
equations
(\ref{pt1}) there exists a unique  TIpGMs for the Potts model on the
Cayley
tree of order $k$ if and only if $z\in\mathcal E_p^{q-1}$.

For some $m=1,\dots,q-1$ we shall consider the following equation
\begin{equation}\label{rm}
z=f_m(z)\equiv \left({(\theta+m-1)z+q-m\over mz+q-m-1+\theta}\right)^k
\end{equation}
\begin{thm} Let $M\subset \{1,\dots,q-1\}$ and $|M|=m$. If $z^*$ is a
solution to
the equation (\ref{rm}),
then $z=(z_1,\dots,z_{q-1})$ is a solution to (\ref{pt1}),
where
\begin{equation}\label{z_i=z}
z_i=\left\{\begin{array}{ll}
1, \ \ \mbox{if} \ \ i\notin M\\[3mm]
z^*, \ \ \mbox{if} \ \ i\in M.
\end{array}
\right.
\end{equation}
\end{thm}

\begin{proof} It is easy to see that $z_i=1$ is a solution of $i$ th
equation of
the system (\ref{pt1}) for each $i=1,2,\dots,q-1$. Consequently,
$z=(1,\dots,1)\in\mathcal E_p^{q-1}$ is a solution to (\ref{pt1}).
Assume that $M\neq\emptyset$. Without loss of generality we can take
$M=\{1,2, \dots, m\}$, $m\leq q-1$. Put $z_i=z,\ i\in M$ and
$z_i=1,\ i\notin M$ from (\ref{pt1}) we get (\ref{rm}).
As $z^*$ is a solution to (\ref{rm}) we get
$z=(\underbrace{z^*,\dots,z^*}_m,1\dots,1)$
is a solution to (\ref{pt1}).
\end{proof}
Now one can ask for what kind of $k$ and prime number $p$ any solution
$z=(z_1,\dots,z_{q-1})$
to (\ref{pt1}) has the form (\ref{z_i=z}). In \cite{RKh} it was proven
that if $k=2$ than for any
prime number $p$ any solution $z\in\mathcal E_p^{q-1}$ to (\ref{pt1})
has the form (\ref{z_i=z}).
In this work we shall consider case $k>2$.
\begin{thm}\label{tti} If $(k^2-k)/2$ is not divisible by $p$ then
any solution
$z\in\mathcal E_p^{q-1}$ to the equation (\ref{pt1}) has the form (\ref{z_i=z}).
\end{thm}
\begin{proof}
For a given $M\subset \{1,\dots,q-1\}$ one can take $z_i=1$ for any
$i\notin M$.
Let $\emptyset\ne M\subset\{1,\dots,q-1\}$, without loss of generality
we can take
$M=\{1,2, \dots, m\}$, $m\leq q-1$, i.e. $z_i=1$, $i=m+1,\dots, q$.
Now we shall
prove that $z_1=z_2=\dots=z_m$.

Let $p>2$. Since $(k^2-k,p)=1$ we get $(k,p)=1$. As
$z_i\in\mathcal E_p,\ i=1,\dots,q-1$ then by Theorem \ref{txq}
there exists $\sqrt[k]{z_i}$ in $\mathbb Q_p$. Denote
$\sqrt[k]{z_i}=x_i$, $i=1,\dots,m$.
Then from (\ref{pt1}) we have
\begin{equation}\label{r1}
x_i=\frac{(\theta-1)x_i^k+\sum_{j=1}^mx_j^k+q-m}
{\sum_{j=1}^mx_j^k+q-m-1+\theta}, \ \ i=1,\dots,m.
\end{equation}
By assumption $x_i\neq1, i=1,2,...,m$ from (\ref{r1}) we get
$$\theta-1=\frac{(x_i-1)\left(\sum_{j=1}^mx_j^k+q-m\right)}{x_i^k-x_i}=
\frac{\sum_{j=1}^mx_j^k+q-m}{x_i(x_i^{k-2}+x^{k-3}_i+\dots+1)}, \ \
i=1,\dots,m.$$
From these equations we get
$$x_i(x_i^{k-2}+x^{k-3}_i+\dots+1)=x_j(x_j^{k-2}+x^{k-3}_j+\dots+1),$$
 which is equivalent to
\begin{equation}\label{x_i-x_j}
\left(x_i-x_j\right)\sum_{n=0}^{k-2}\sum_{s=0}^{n}x_i^{n-s}x_j^s=0
\end{equation}
As $x_i,x_j\in\mathcal E_p$ and $(k^2-k,p)=1$ we have
$$\left|\sum_{n=0}^{k-2}\sum_{s=0}^{n}x_i^{n-s}x_j^s\right|_p=
\left|k^2-k\right|_p=1.$$
Hence,
$x_i=x_j$ for any $i,j\in \{1,\dots,m\}$.

Let $p=2$. Since $k^2-k\not\equiv 0(\operatorname{mod} 4)$
we get $k=4n+2$ or $k=4n+3$. It is clear that
$(k,2)=1$ if $k=4n+3$. In this case by Theorem \ref{txq}
there exists $\sqrt[k]{z_i},\ i=1,\dots,m$ in
$\mathbb Q_2$. If $k=4n+2$ then by Theorem \ref{tx2} there
exists $\sqrt[k]{z_i}$.
Denote $x_i=\sqrt[k]{z_i}$. Note that $\left|x_i-1\right|_2
\leq\frac{1}{2}$. By $p$-adic norm's property we obtain
$$\left|\sum_{n=0}^{k-2}\sum_{s=0}^{n}x_i^{n-s}x_j^s\right|_p=
\left|\frac{k^2-k}{2}\right|_p=1.$$
From (\ref{x_i-x_j}) we get $x_i=x_j$ for any $i,j\in \{1,\dots,m\}$.
\end{proof}

By this theorem we have that any TIpGM of the Potts model
corresponds to a solution of the (\ref{rm}).

\begin{rk}
Note that the condition $(\frac{k^2-k}{2},p)=1$ is
a sufficient to any solution of the equation (\ref{pt1}) has
the form (\ref{z_i=z}). In \cite{RKh} it was shown that if
$k=p=3$ then the equation (\ref{pt1}) has a
solution $z\in\mathcal E_3^{q-1}$ which is not represented in
the form (\ref{z_i=z}).
\end{rk}

Let $M\subset \{1,\dots, q-1\}$, with $|M|=m$. Then corresponding
solution of (\ref{rm})
we denote by $z(M)=\exp_p(h(M))$. It is clear that $h(M)=h(m)$, i.e.
it only depends on cardinality of $M$.
Put
$${\mathbf 1}_M=(e_1,\dots,e_q), \ \ \mbox{with} \ \ e_i=1 \ \
\mbox{if} \ \ i\in M, \ \ e_i=0 \ \ \mbox{if} \ \ i\notin M.$$

We denote by $\mu_{h(M)\mathbf 1_M}$ the TIpGMs
corresponding to the solution $h(M)$.

The following proposition is useful.

\begin{pro}\label{tp} For any finite $\Lambda\subset V$ and
any $\sigma_\Lambda\in \{1,\dots,q\}^\Lambda$ we have
\begin{equation}\label{mu}
\mu_{h(M){\mathbf 1}_M}(\sigma_\Lambda)=\mu_{h(M^c)
{\mathbf 1}_{M^c}}(\sigma_\Lambda),
\end{equation}
where $M^c=\{1,\dots,q\}\setminus M$ and $h(M^c)=-h(M)$.
\end{pro}
Proof is similar to the proof of the Proposition 1 in \cite{KRK}.
The following is a corollary of Theorem \ref{tti} and Proposition \ref{tp}.
\begin{cor}\label{cor1} Let $\frac{k^2-k}{2}\not\equiv0(\operatorname{mod }p)$.
Then each TIpGMs corresponds to a solution of (\ref{rm})
with some $m\leq [q/2]$,
where $[a]$ is the integer part of $a$.
Moreover, for a given $m\leq [q/2]$,
a fixed solution to (\ref{rm}) generates ${q\choose m}$ vectors
${\tilde h}$ giving ${q\choose m}$ TIpGMs.
\end{cor}

\subsection{\bf Case $k=3$} In this section we consider the following
$$
f_m(z)=\left({(\theta+m-1)z+q-m\over mz+q-m-1+\theta}\right)^3
$$
From $\frac{f_m(z)-z}{z-1}=0$ we get
\begin{equation}\label{kubik}
z^3-Az^2-Bz+C=0,
\end{equation}
where
\begin{equation}\label{ABC}
\begin{array}{ll}
A=\frac{-3m^2(q-m)+3m(\theta-1)^2+(\theta-1)^3}{m^3},\\[3mm]
B=\frac{-3m(q-m)^2+3(q-m)(\theta-1)^2+(\theta-1)^3}{m^3},\\[3mm]
C=\frac{(q-m)^3}{m^3}.
\end{array}
\end{equation}

Note that $z=1$ is a solution to the equation (\ref{kubik}) if and
only if $\theta=1-q$ or $\theta=1+q/2$. Moreover,
 if $\theta=1-q$ then the equation (\ref{kubik}) has three solutions
 $z_1=z_2=1$ and
 $z_3=\left(\frac{m-q}{m}\right)^3$ in $\mathbb Q_p$. Note that
 $z_3\neq1$ as $m<q$. Recall that there exists
 a $p$-adic Gibbs measure $\mu$ corresponding to the vector
 $(\underbrace{z_3,\dots,z_3}_m,1,\dots,1)\in\mathbb Q_p^{q-1}$ if
 and only if $z_3$
 belongs to $\mathcal E_p$.

\begin{pro}\label{theta=1-q} Let $\theta=1-q$.\\
$A)$ Then $z_3\in\mathcal E_p$ if at least one of the following
conditions is satisfied\\
$(A.1)$ $p=2$ and $|q|_2<|2m|_2$;\\
$(A.2)$ $p\neq2$ and $|q|_p<|m|_p$;\\
$(A.3)$ $p\neq2$, $|q|_p=|m|_p$ and $x^2+x+1\equiv 0(\operatorname{mod} p)$,
where $x=1-q/m$;\\
$B)$ Otherwise, $z_3\notin\mathcal E_p$.
\end{pro}

\begin{proof} We must check $\left|z_3-1\right|_p<p^{-1/(p-1)}$ which is
equivalent to
\begin{equation}\label{q/m1}
\left|\frac{q}{m}\left(2-\frac{q}{m}+\left(1-\frac{q}{m}\right)^2\right)
\right|_p<\frac{1}{2},\qquad\mbox{if }\ p=2
\end{equation}
and
\begin{equation}\label{q/m2}
\left|\frac{q}{m}\left(2-\frac{q}{m}+\left(1-\frac{q}{m}\right)^2\right)
\right|_p<1,\qquad\mbox{if }\ p\neq2
\end{equation}

Let $p=2$. It is easy to see that $|z_3-1|_2<\frac{1}{2}$ if $|q|_2<|2m|$.
If $|q|_2\geq|2m|_2$ from (\ref{q/m1})
 we get
 $$
 \left|\frac{q}{m}\left(2-\frac{q}{m}+\left(1-\frac{q}{m}\right)^2\right)
 \right|_p\geq\frac{1}{2}.
 $$
 This means that $z_3\notin\mathcal E_2$.

Let $p\neq2$. If $|q|_p<|m|_p$ then by $p$-adic norm's
property we have (\ref{q/m2}).
One can see that (\ref{q/m2}) is not satisfied if $|q|_p>|m|_p$.
Assume that $|q|_p=|m|_p$.
Denote $x=1-q/m$.
We have
$$
\left|\frac{q}{m}\left(2-\frac{q}{m}+\left(1-\frac{q}{m}\right)^2\right)
\right|_p=\left|1+x+x^2\right|_p
$$
Consequently, in this case the condition (\ref{q/m2}) is satisfied
if and only if
$x^2+x+1\equiv 0(\operatorname{mod} p)$.
 \end{proof}

By Proposition \ref{theta=1-q} and Corollary \ref{cor1}
we get the following

\begin{thm} Let $\theta=1-q$. Then $\left|\mathcal G(H)
\right|\geq1+{q\choose m}$
if one of the following
conditions is satisfied\\
$a)$ $p=2$ and $|q|_2<|2m|_2$;\\
$b)$ $p\neq2$ and $|q|_p<|m|_p$;\\
$c)$ $p\neq2$, $|q|_p=|m|_p$ and $x^2+x+1\equiv 0(\operatorname{mod} p)$,
where $x=1-q/m$.
\end{thm}

Now we shall consider the case $\theta=1+q/2$. Dividing the equation
(\ref{kubik}) to $z-1$ we obtain
\begin{equation}\label{kv}
8m^3z^2-\left(16m^3-24m^2q+6mq^2+q^3\right)z-8(q-m)^3=0.
\end{equation}

Denote $D=q^2+12mq-12m^2$ and $f(z)$ is a quadratic function (\ref{kv}).
Note that $D\neq0$. So, the equation (\ref{kv})
has two distinct solutions $z_2$ and $z_3$
in $\mathbb Q_p$ if $\sqrt{D}$ exists. It easy to see that $z_2=1$
and $z_3=-1$ are
solutions of (\ref{kv}) if $q=2m$. One can see that
$z_3\notin\mathcal E_p$.
We assume that $q\neq2m$ to
find a solution of (\ref{kv})
belongs to $\mathcal E_p\setminus\{1\}$.

\begin{pro}\label{theta=1+q}
Let $\theta=1+q/2$. Then it holds the following statements:\\
$A)$ The equation (\ref{kv}) has two solutions in $\mathcal E_p$
if $|m|_p>|q|_p$ and $\sqrt{-3}$ exists in $\mathbb Q_p$;\\
$B)$ The equation (\ref{kv}) has a unique solution in $\mathcal E_p$
if one of the following conditions is satisfied\\
$(B.1)$ $p=2$ and $\left|\frac{q}{2m}-1\right|_p<\frac{1}{4}$;\\
$(B.2)$ $p=3$ and $|m|_p\leq|q|_p$;\\
$(B.3)$ $p>3$, $|m|_p=|q|_p$ and $|q-2m|_p<|q|_p$;\\
$C)$ Otherwise the equation (\ref{kv}) has no solution in $\mathcal E_p$.
\end{pro}

\begin{proof}
{\bf Case $|m|_p<|q|_p$}. If $p\neq3$ then by $p$-adic norm's property
one finds
$$
|f(1)|_p=\left|9q^2(q-2m)\right|_p=
\left|q^3\right|_p>0.
$$
This means that if $p\neq3$ than (\ref{kv}) has no solution
in $\mathcal E_p$.

Let $p=3$. From $f^\prime(1)=24m^2q-6mq^2-q^3$ we get
$\left|f^\prime(1)\right|_p=\left|q^3\right|_p$.
Thus, we show that for the function $q^{-3}f(z)$ all
conditions of
Hensel's lemma are satisfied. By Hensel's lemma the
equation (\ref{kv})
has a unique solution in $\mathcal E_3$.

{\bf Case $|m|_p>|q|_p$}. In this case we get
$z_2, z_3\in\mathbb Z_p^*$ if
$\sqrt{D}$ exists. At first, we assume that $p=2$. Then we have
$$
D=16m^2\left(1+2\alpha+\frac{\alpha^2}{4}\right),\ \mbox{where }\ \alpha=\frac{q}{2m}-1.
$$
Note that $|\alpha|_2\leq1$. Then by Theorem \ref{tx3}
a number $\sqrt{D}$ exists in $Q_2$ iff
$\left|\alpha\right|_2<\frac{1}{4}$.
Assume that $|\alpha|_2<\frac{1}{4}$. We obtain
\begin{equation}\label{z2z3}
z_2\cdot z_3=\left(\frac{m-q}{m}\right)^3=\left(-1-2\alpha\right)^3
=1+2+4+\cdots
\end{equation}
Hence, one can find $z_2\notin\mathcal E_2$ or $z_3\notin\mathcal E_2$.
Assume that $z_2\notin\mathcal E_2$, i.e. $z_2\equiv3(\operatorname{mod }4)$.
Then from (\ref{z2z3}) we have $z_3\equiv1(\operatorname{mod }4)$, which is
equivalent to $z_3\in\mathcal E_2$.

Let $p\neq2$.
Then we have
$$
D=4m^2\left(-3+6\beta+\beta^2\right),\ \mbox{where }\ \beta=\frac{q}{2m}.
$$
Then by Theorem \ref{tx3} existence of $\sqrt{D}$ is
equivalent to the existence of $\sqrt{-3}$. Assume that
$\sqrt{-3}\in\mathbb Q_p$.
Then the equation (\ref{kv}) has two
solutions $z_2$ and $z_3$ in
$\mathbb Z_p^*$. We get the followings
$$
z_2+z_3\equiv2(\operatorname{mod} p)\qquad\mbox{and }
\qquad z_2\cdot z_3\equiv1(\operatorname{mod }p).
$$
Hence, we obtain $z_i^2+1\equiv2z_i(\operatorname{mod} p), i=2,3$
which is equivalent to
$z_i\equiv1(\operatorname{mod }p)$.
This means $z_2, z_3\in\mathcal E_p$.

{\bf Case $|m|_p=|q|_p$}. Assume that $|q-m|_p<|q|_p$. Note
that this inequality automatically
executed when $p=2$. By $p$-adic norm's property we have
$$
|f(1)|_p=
\left|9q^2(q-2m)\right|_p=
\left|q^3\right|_p>0,\ \mbox{if } p\neq3.
$$
Hence, we conclude that (\ref{kv}) has no solution in
$\mathcal E_p$ for $p\neq3$.
Note that if $|q-m|_3<|q|_3$ then by Hensel's lemma the
equation (\ref{kv}) has a unique solution in $\mathcal E_3$.

Let $|q-2m|_p<|q|_p$. In this case we have $\left|q^{-3}f(1)\right|_p<1$
and $\left|q^{-3}f^\prime(1)\right|_p=1$. Then by Hensel's lemma (\ref{kv})
has a unique solution in $\mathcal E_p$.

Let $|q-m|_p=|q-2m|_p=|q|_p$. Since $|q|_p=|m|_p$ we get $p>3$.
Then we have $|f(1)|_p=\left|q^3\right|_p>0$. This means that (\ref{kv})
has no solution in $\mathcal E_p$.
\end{proof}

By Proposition \ref{theta=1+q} and Corollary \ref{cor1} we get the following

\begin{thm}
Let $\theta=1+q/2$. Then $\left|\mathcal G(H)
\right|\geq1+{q\choose m}$ if one of the following
conditions is satisfied\\
$a)$ $p=2$ and $\left|\frac{q}{2m}-1\right|_p<\frac{1}{4}$;\\
$b)$ $p=3$ and $|m|_p\leq|q|_p$;\\
$c)$ $p>3$, $|m|_p=|q|_p$ and $|q-2m|_p<|q|_p$;\\
$d)$ $|m|_p>|q|_p$ and $\sqrt{-3}$ exists in $\mathbb Q_p$;
\end{thm}

Now we assume that $\theta\notin\{1-q, 1+q/2\}$ and consider the
following equation
\begin{equation}\label{kubik2}
x^3+\alpha x+\beta=0,
\end{equation}
which is obtained after the change of variable $z=x+A/3$.
Here
\begin{equation}\label{alpha}
\alpha=-\left(B+\frac{A^2}{3}\right),\qquad \beta=
-\frac{2}{27}A^3-\frac{AB}{3}+C.
\end{equation}

Note that if $|\theta-1|_p\neq|q|_p$ and
$|m|_p<\max\{|\theta-1|_p,|q|_p\}$ then the equation
(\ref{kubik}) has no solution in $\mathcal E_p$. So, we assume that
$|m|_p>\max\{|\theta-1|_p,|q|_p\}$ to
find nontrivial
$p$-adic Gibbs measures.
Under this condition, by $p$-adic norm's property
from (\ref{ABC})
and (\ref{alpha}) one finds $|\alpha|_p=
\left(\frac{|\theta-1|_p}{|m|_p}\right)^2$ and
$|\beta|_p=\left(\frac{|\theta-1|_p}{|m|_p}\right)^3$ if $p>3$.
\begin{lemma}\label{z,x} Let $q\in p\mathbb N$ and $|m|_p>
\max\{|\theta-1|_p,|q|_p\}$.
Then a solution $z$ of (\ref{kubik}) belongs to $\mathcal E_p$ iff
the corresponding solution $x=z-A/3$ of (\ref{kubik2}) belongs to
$\mathbb Z_p\setminus\mathbb Z_p^*$.
\end{lemma}
\begin{proof} Since $|m|_p>\max\{|\theta-1|_p,|q|_p\}$ we obtain
$$\left|\frac{A}{3}-1\right|_p=\frac{\left|(\theta-1)^3+
3(\theta-1)^2m-3m^2q\right|_p}{\left|3m^3\right|_p}
\leq\frac{\max\{\left|(\theta-1)^2\right|_p,|mq|_p\}}
{\left|m^2\right|_p}<1.$$
Hence, by $p$-adic norm's property it hold true the
following inequality
$$|z-1|_p=\left|x+\frac{A}{3}-1\right|_p<1$$
if and only if $|x|_p<1$.
\end{proof}
\begin{pro}\label{thetay1-q} Let $p>3$ and $\theta\notin\{1-q, 1+q/2\}$.
Then for a given $m=1,2,\dots,[q/2]$ the equation (\ref{kubik2})
has three solutions in $\mathbb Z_p\setminus\mathbb Z_p^*$ if
$\max\left\{\left|\frac{\theta-1}{m}\right|_p,\left|\frac{q}{m}\right|_p\right\}\leq p^{-3}$.
\end{pro}
Proof is similar to the proof of Theorem 4.3
in \cite{MFHA}. By Lemma \ref{z,x} and
Proposition \ref{thetay1-q}
we get the following

\begin{thm} Let $p>3$ and $\theta\notin\{1-q, 1+q/2\}$. 
If there exists $m=1,2,\dots,[q/2]$ such that  
$\max\left\{\left|\frac{\theta-1}{m}\right|_p,\left|\frac{q}{m}\right|_p\right\}\leq p^{-3}$ 
then there exist at last 
$1+3{q\choose m}$ TIpGMs for the $p$-adic Potts model on a Cayley tree of order three.
\end{thm}

\section{boundedness of translation-invariant $p$-adic gibbs measures}
Now we shall study the problem of boundedness of translation-invariant $p$-adic Gibbs
measures. Note that if $q\notin p\mathbb N$ then by Theorem \ref{qpn}
there exists only one translation-invariant $p$-adic Gibbs measure $\mu_0$. In \cite{MRasos}
 it have been proven that $p$-adic Gibbs measure $\mu_0$ is bounded if and only if
 $q\notin p\mathbb N$. Assume that $m\in\{1,2,...,[q/2]\}$ and
 $z(m)\in\mathcal E_p\setminus{1}$ is a solution to the equation (\ref{rm}).
 We shall show that corresponding $p$-adic Gibbs measure $\mu_{h(m)}$ is
 not bounded. Since
$$\left|\mu_{h(m)}^{(n)}(\sigma)\right|_p=\frac{\left|\exp_p\left(H_n(\sigma)
+\sum_{x\in W_n}h(m){\bf1}(\sigma(x)\leq m)\right)\right|_p}{\left|Z_{n,h(m)}\right|_p}
=\frac{1}{\left|Z_{n,h(m)}\right|_p}$$
We shall show that
$$\left|Z_{n,h(m)}\right|_p\to 0,\qquad n\to\infty.$$
For the normalizing constant we have the following recurrence formula \cite{MRasos}
\begin{equation}\label{rec}
Z_{n+1,h}=A_{n,h}Z_{n,h},\qquad\mbox{where}\ \ A_{n,h}=\prod_{x\in W_n}a_h(x).
\end{equation}
For the solution $z(m)\in\mathcal E_p\setminus\{1\}$ to the equation (\ref{rm}) we have
$$a_{h(m)}(x)=(m(z(m)-1)+q+\theta-1)^k,\qquad\mbox{where}\ \ z(m)=\exp_p(h(m)).$$
Then by (\ref{rec}) we get
$$Z_{n,h(m)}=(m(z(m)-1)+q+\theta-1)^{k|V_{n-1}|}.$$
From this considering $|z(m)-1|_p<1,\ |\theta-1|_p<1$ and $q\in p\mathbb N$ we have
$$\left|Z_{n,h(m)}\right|_p<p^{-k|V_{n-1}|}.$$
Hence,
$$\left|Z_{n,h(m)}\right|_p\to 0,\qquad n\to\infty.$$
Thus, we have proved the following
\begin{thm} A translation-invariant $p$-adic Gibbs measure for the Potts model on the
Cayley tree of order $k$ is bounded if and only if $q\notin
p\mathbb N$.
\end{thm}

\end{document}